%% file: main.tex
\documentclass{article}

\usepackage{arxiv}

\usepackage[utf8]{inputenc}
\usepackage[T1]{fontenc}
\usepackage{hyperref}
\usepackage{url}
\usepackage{doi}

\usepackage[english]{babel}
\usepackage{amsmath}
\usepackage{amsfonts}
\usepackage{graphicx}
\usepackage{natbib}
\usepackage{enumitem}

\usepackage{amssymb}

\usepackage[misc,geometry]{ifsym}
\usepackage{pifont}
\usepackage{geometry}
\usepackage{graphicx}

\usepackage{algorithm}
\usepackage[noend]{algpseudocode}
\usepackage{mathtools}
\usepackage{relsize}

\usepackage{array}
\usepackage{multirow}
\usepackage{dcolumn}
\usepackage{color}
\usepackage{url}
\usepackage{mathrsfs}
\usepackage{bbm}
\usepackage{subcaption}
\usepackage{setspace}
\usepackage[table]{xcolor}
\usepackage{tikz}
\usepackage{longtable}
\usepackage{lscape}
\usepackage{multirow}

\newtheorem{prop}{Proposition}{\bf}{\it}
\newtheorem{definition}{Definition}{\bf}{\it}
\newtheorem{proof}{Proof}{\it}{\rm}

\usetikzlibrary{knots, positioning}

\allowdisplaybreaks

\title{Efficient labeling algorithms for adjacent quadratic shortest paths}

\author{
\href{https://orcid.org/0000-0002-2097-4131}{\includegraphics[scale=0.06]{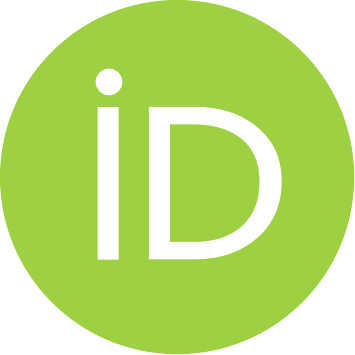}\hspace{1mm}João Vilela} \\
Departamento de Engenharia Industrial \\
Pontif\'{\i}cia Universidade Cat\'olica do Rio de Janeiro \\
\texttt{jvilela@aluno.puc-rio.br} \\
\And
\href{https://orcid.org/0000-0003-0104-3419}{\includegraphics[scale=0.06]{orcid.pdf}\hspace{1mm}Bruno Fanzeres} \\
Departamento de Engenharia Industrial \\
Pontif\'{\i}cia Universidade Cat\'olica do Rio de Janeiro \\
\texttt{bruno.santos@puc-rio.br} \\
\And
\href{https://orcid.org/0000-0001-5715-149X}{\includegraphics[scale=0.06]{orcid.pdf}\hspace{1mm}Rafael Martinelli} \\
Departamento de Engenharia Industrial \\
Pontif\'{\i}cia Universidade Cat\'olica do Rio de Janeiro \\
\texttt{martinelli@puc-rio.br} \\
\And
\phantom{xxxxxxxxxxxxxx}\href{https://orcid.org/0000-0001-7595-3904}{\includegraphics[scale=0.06]{orcid.pdf}\hspace{1mm}Claudio Contardo} \phantom{xxxxxxxxxxxxxx}\\
UQAM School of Business, \\
CIRRELT and GERAD \\
\texttt{claudio.contardo@gerad.ca}
}

\date{}

\renewcommand{\headeright}{A Preprint}
\renewcommand{\undertitle}{A Preprint}
\renewcommand{\shorttitle}{Efficient labeling algorithms for AQSPP}

\hypersetup{
pdftitle={Efficient labeling algorithms for adjacent quadratic shortest paths},
pdfsubject={math.OC},
pdfauthor={João Vilela, Bruno Fanzeres, Rafael Martinelli, Claudio Contardo},
pdfkeywords={Efficient labeling algorithms for adjacent quadratic shortest paths},
}

\begin{document}
\maketitle

\begin{abstract}
In this article, we study the Adjacent Quadratic Shortest Path Problem (AQSPP), which consists in finding the shortest path on a directed graph when its total weight component also includes the impact of consecutive arcs. We provide a formal description of the AQSPP and propose an extension of Dijkstra's algorithm (that we denote \texttt{aqD}) for solving AQSPPs in polynomial-time and provide a proof for its correctness under some mild assumptions. Furthermore, we introduce an adjacent quadratic A\textsuperscript{*} algorithm (that we denote \texttt{aqA\textsuperscript{*}}) with a backward search for cost-to-go estimation to speed up the search. We assess the performance of both algorithms by comparing their relative performance with benchmark algorithms from the scientific literature and carry out a thorough collection of sensitivity analysis of the methods on a set of problem characteristics using randomly generated graphs. Numerical results suggest that: (i) \texttt{aqA\textsuperscript{*}} outperforms all other algorithms, with a performance being about 75 times faster than \texttt{aqD} and the fastest alternative; (ii) the proposed solution procedures do not lose efficiency when the magnitude of quadratic costs vary; (iii) \texttt{aqA\textsuperscript{*}} and \texttt{aqD} are fastest on random graph instances, compared with benchmark algorithms from scientific literature. We conclude the numerical experiments by presenting a stress test of the AQSPP in the context of real grid graph instances, with sizes up to $16 \times 10^6$ nodes, $64 \times 10^6$ arcs, and $10^9$ quadratic arcs.
\end{abstract}



\section{Introduction} \label{SecI_introduction}

    Finding the least-weighted path between two vertices in a graph is a well-studied problem in computer science and has been deeply explored in the technical literature both from a methodological perspective \citep{pollack1960, ahuja1993_NetworkFlow, madkour2017} as well as from an applied point of view \citep{Duque2015_BiObjectiveSPP, Hougardy2017_DijkstraMeetsSteiner, Furini2020_Int_Bilevel_kVertex, Russ2021_ConstrRelSPP_TimeDepNet}. This problem is usually referred to as the \textit{Shortest Path Problem} (SPP) and aims at finding a path $P$ on a weighted graph $G = (V,A, \boldsymbol{c})$ that connects a source node $s \in V$ with a target node $t \in V$ of minimal total weight. Structurally, the SPP can be formulated as an integer programming problem with a linear objective function $\boldsymbol{c}^{\top} \mathbf{x}$, where $\boldsymbol{c} \in \mathbb{R}^{|A|}$ represents an edges-wise weight vector and $\mathbf{x} \in \{0, 1\}^{|A|}$ indicates a binary vector of arc selections, constrained to a network-flow-based feasible region, that builds up the shortest path \citep{Taccari2016_IPForm_ESPP}. Many procedures have been developed over the past decades to solve the SPP efficiently (e.g., Bellman-Ford \citep{bellman1958}, Floyd-Warshall \citep{floyd1962}, and A\textsuperscript{*} \citep{hart1968} algorithms), with the most well-known approach due to Dijkstra \citep{dijkstra1959}.

    Although the linear nature of SPPs can accommodate a large collection of problem structures, some applications demand more complex interactions and cost-impacts between edges within the path, as, for instance, the problems of establishing network protocols \citep{murakami1997}, optimizing vehicle routing \citep{martinelli2015}, and identifying variance-constrained shortest paths \citep{Sivakumar1994_VarConstSPP, Sen2001_MeanVarTravel}, all intrinsically following a quadratic order. In this context, the \textit{Quadratic Shortest Path Problem} (QSPP) emerges as a powerful modeling tool to properly characterize a broader class of problems since it extends the standard SPP structure by identifying a path that jointly minimizes both the total path weight and the sum of interaction costs over distinct pairs of edges on the path. Formally, the QSPP can be stated over a weighted digraph $G = (V,A,\boldsymbol{c},\boldsymbol{q})$ as a quadratic integer programming problem \citep{Gill2015_MethGenQP} with objective function defined as $\mathbf{x}^{\top} \boldsymbol{q} \mathbf{x} + \boldsymbol{c}^{\top} \mathbf{x}$, where $\boldsymbol{q} \in \mathbb{R}^{|A|} \times \mathbb{R}^{|A|}$ is a square matrix that maps the weight impact of pair of arcs and $(\boldsymbol{c}, \mathbf{x})$ follows similarly to the SPP formulation \citep{rostami2015}.
%
    In this work, we focus on a particular structure of QSPPs, referred to in technical literature as the \textit{Adjacent Quadratic Shortest Path Problem} (AQSPP), which follows the same quadratic structure, but in which the matrix $\boldsymbol{q}$ only accounts for interactions among adjacent arcs. 

    The theory and applicability of QSPP/AQSPP have been studied through the seminal work of \cite{Amaldi2011_ReloadCost} and further extended by \cite{rostami2015, rostami2018}, \cite{hu2017, hu2018} and \cite{buchheim2015}. \cite{Amaldi2011_ReloadCost} investigated the issues over complexity and approximability in problems that seek for optimum paths, tours, and flows under the so-called reload costs, defined as the cost incurred by the path when two adjacent arcs are of ``different types''. Moreover, \cite{buchheim2015} explores the more general framework of \textit{Binary Quadratic Problems} (BQPs) and propose an exact algorithm for this class of optimization problems based on computing quadratic global underestimators and embedding the resulting lower bounds into branch-and-bound techniques. \cite{rostami2015} provides an in-depth mathematical description of QSPPs and proposes an algorithm to solve AQSPPs in polynomial-time based on a procedure that reduces the AQSPP to an instance of the SPP by modifying the original graph. Furthermore, \cite{hu2018} also characterize QSPPs mathematically and derive an algorithm that examines whether QSPP instances of directed grid graphs are ``linearizable''. More precisely, the authors evaluate when a grid graph with quadratic cost can be transformed into a graph with linear costs. They also reviewed the developments presented by Rostami et al. in \cite{rostami2015} and prove that, despite stated in \cite{rostami2015}, their algorithm does not find optimal solutions in polynomial time for general digraphs, but rather only for the particular case of acyclic digraphs.


    Therefore, following this latter observation and technical issue, in this work, we study computationally efficient solution methods to solve AQSPPs on a broader class of graphs, the ones that do not possess improving $\alpha$-cycles, defined as follows. 
    
    \begin{definition}
    An $\alpha$-cycle is a walk $(u_1, \ldots, u_n)$ in the underlying graph with $n \geq 5$ such that $u_2 = u_{n - 1}$. An $\alpha$-cycle is said to be \textit{improving} if its total compounded weight is lower than that of the shorter sequence $(u_1, u_2, u_n)$.
    \end{definition}
    
    Note that an improving $\alpha$-cycle may exist even without negative weights due to their quadratic nature. In Figure \ref{fig:alphacyc}, we depict an $\alpha$-cycle of five nodes. The $\alpha$-cycle depicted in this figure would be improving if, for instance, one had a substantial quadratic weight impact for the sequence $(1, 2, 5)$ and very low for all other ones.
    We argue that graphs with no improving $\alpha$-cycles are more general than acyclic ones since, by definition, the latter do not possess $\alpha$-cycles, but the opposite does not hold. These types of graphs are found in several practical applications, as, for instance, on graphs that represent spatial attributes (terrain, distance, slope, etc.), wildly used on route-design problems \citep{wan2011, ebrahimipoor2009, Barmann2015_SolveNetworkDsgnProb, piveteau2017}.
    
    \begin{figure}[htbp]
    	\centering
    \begin{tikzpicture}
    	\tikzset{main node/.style={circle,fill=blue!20,draw,minimum size=0.5cm,inner sep=1pt},
    	}
    	\node[main node] (1) {$1$};
    	\node[main node] (2) [below left = of 1]  {$2$};
    	\node[main node] (3) [below left = of 2] {$3$};
    	\node[main node] (4) [above left = of 2] {$4$};
    	\node[main node] (5) [below right = of 2] {$5$};
    	
    	\path[draw,thick,->]
    	(1) edge node {} (2)
    	(2) edge node {} (3)
    	(3) edge node {} (4)
    	(4) edge node {} (2)
    	(2) edge node {} (5);
    \end{tikzpicture}
\caption{An $\alpha$-cycle of five nodes\label{fig:alphacyc}}
	\end{figure}
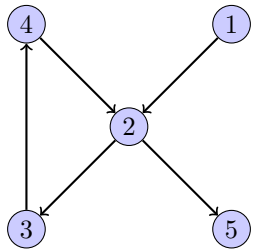

This work aims to contribute to the scientific literature on AQSPPs by designing an extension of Dijkstra's algorithm (from now on referred to as \texttt{aqD}) to solve AQSPP on graphs with no improving $\alpha$-cycles. More specifically, we demonstrate that our algorithm finds optimal walks (paths possibly with cycles) in the absence of negative weights and acyclic paths when the graph is free of improving $\alpha$-cycles. An acceleration to the algorithm is also proposed: an adjacent quadratic A\textsuperscript{*} with a backward search for cost-to-go estimation (hereinafter referred to as \texttt{aqA\textsuperscript{*}}). We perform a thorough computational campaign to analyze the proposed algorithms' performances. First, we measure the relative speed of \texttt{aqD} and \texttt{aqA\textsuperscript{*}} against the linearization-based algorithm designed in \cite{rostami2015, rostami2018}. Next, we analyze the sensitivity of our methods to variations in the quadratic weight levels. Then, we empirically evaluate the performance of our two algorithms on randomly generated graphs using the Erdős-Rényi \citep{erdos-renyi1959, Bollobas2001} and Configuration Model \citep{newman2010} methods. The instances used for these computational tests range from $6 \times 10^5$ to $6.4 \times 10^7$ quadratic arcs. Finally, we stress the search process by increasing the graph sizes up to $16 \times 10^6$ nodes, $64 \times 10^6$ arcs, and $10^9$ quadratic arcs.

The outline of this work is as follows. In the next section, we present the formal problem definition with a quadratic formulation for the AQSPP. In Section \ref{sec:aqspp}, we introduce the extension of the classical Dijkstra's algorithm for the adjacent quadratic case and provide a proof for its correctness in the absence of improving $\alpha$-cycles. Section \ref{sec:astar} extends the proposed solution procedure using a cost-to-go estimation to obtain an A\textsuperscript{*} algorithm. Section \ref{sec:on-the-fly} describes how to speed up the algorithm even further by calculating the quadratic costs on the fly to handle vary-large scale AQSPPs instances. In Section \ref{sec:experiments}, we compare the proposed algorithms against benchmark approaches from scientific literature using different graph types. Finally, Section \ref{sec:conclusion} concludes this work and points the avenues of future research.

    
\section{Problem Definition} \label{sec:definition}

    The AQSPP is defined over a weighted digraph $G = (V, A, \boldsymbol{c}, \boldsymbol{q})$, where $(V, A)$ denote the sets of nodes and arcs of the network, respectively, $\boldsymbol{c} \in \mathbb{R}_+^{|A|}$ stands for the vector of linear costs, and $\boldsymbol{q} \in \mathbb{R}_+^{|A| \times |A|}$ represents the respective matrix of quadratic costs, where we assume that $q(a_1, a_2) = 0$ if $h(a_1) \neq t(a_2)$, where the functions $t(\cdot)$ and $h(\cdot)$ stand for the \textit{head} and the \textit{tail} of an arc, respectively. This assumption allows us to overload the notation and denote $q(u, v, w)$ instead of $q((u, v), (v, w))$ for every pair of consecutive arcs in $A\times A$. In addition to $G$, we are also given a source node $s\in V$ and a target node $t\in V\setminus\{s\}$. For every node $u\in V$ we denote $\delta^+(u) = \{a\in A: t(a) = u\}$ and $\delta^-(u) = \{a\in A: h(a) = u\}$.
    
    The AQSPP can be formulated as a binary quadratic program (BQP), as presented in \eqref{eq:AQSPP_FO}--\eqref{eq:aqspp:bin}. We are given, for every arc $a\in A$, a binary variable $x_a$ that takes the value 1 if $a$ is part of the shortest path. The following BQP is valid for the AQSPP:
    \begin{equation}
        \min_{\mathbf{x}} \quad \sum_{a\in A} c(a)x_a + \sum_{(a, b)\in A \times A}q(a, b)x_ax_b \label{eq:AQSPP_FO}
    \end{equation}
    subject to
    \begin{align}
        \sum_{a\in\delta^+(u)}x_a - \sum_{a\in\delta^-(u)}x_a &= \begin{cases}
                                                                    1 & \hspace{0.50cm} u = s\\
                                                                    -1 & \hspace{0.50cm} u = t\\
                                                                    0 & \hspace{0.50cm} u \in V \setminus \{s, t\}
                                                                    \end{cases} & u\in V\label{eq:aqspp:flow}\\
        x_a&\in\{0, 1\} & a\in A.\label{eq:aqspp:bin}
    \end{align}
    
    Note that the feasible region defined by \eqref{eq:aqspp:flow}-\eqref{eq:aqspp:bin} on the AQSPP formulation is equivalent to the standard SPP \citep{Taccari2016_IPForm_ESPP}. Thus, all paths that are feasible for one of them, are also for the other. The main difference between the linear and the quadratic case comes from introducing the quadratic term $q(a, b) x_{a} x_{b}$ in the objective function \eqref{eq:AQSPP_FO}. Due to the total unimodularity characteristic of the constraint matrix, one may be tempted to drop the integrality constraints for the binary variables. However, note that unlike the linear case, an interior point may now be optimal in the AQSPP, justifying the need for the integrality constraints \eqref{eq:aqspp:bin} for the $\mathbf{x}$-variables. 


It should be highlighted that using BQP techniques to solve the AQSPP seems a reasonable approach and has already been explored in the scientific literature \citep{buchheim2015,caprara2008}. However, we argue that as the size of the graph grows, using off-the-shelf solvers without exploiting the underlying shortest path structure of the problem may quickly become prohibitive. Memory issues that raise from building up a very heavy optimization problem can easily occur and prevent the solver from providing meaningful results. Additionally, the performance of BQP solvers in large instances can quickly deteriorate. These limitations motivate us to study this problem from an algorithmic perspective, to ultimately introduce two algorithms to solve AQSPPs. Both algorithms are described in Sections \ref{sec:aqspp} and \ref{sec:astar}.


\section{Adjacent Quadratic Dijkstra} \label{sec:aqspp}

In this section, we describe a dynamic programming algorithm for the AQSPP, referred to in this work as the \textit{adjacent quadratic Dijkstra} (\texttt{aqD}). The algorithm is based on the concept of node labels and can adequately address adjacent costs.

\subsection{The Algorithm}

The Adjacent Quadratic Dijkstra algorithm proposed in this work follows a similar structure to the standard Dijkstra's algorithm, but computing the shortest distances between the source and every possible arc in the network, as opposed to the standard Dijkstra solution method that computes the minimum costs between the source and every possible node in the network. More specifically, for a given source node $s \in V$ and an arc $a = (i, j)\in A$, we define a $s$-$a$-walk as a tuple $(v_0,v_1,\ldots, v_k)$ such that $v_0 = s, v_{k - 1} = i, v_k = j$, and let us to denote by $D_a$ the cost of the best $s$-$a$-walk. This notation allows computing adjacent costs when extending a walk that ends at arc $a$. 
We present a pseudo-code of the solution method in Algorithm \ref{alg:aqdijkstra}.
    \begin{algorithm}[htbp]
        \caption{Adjacent Quadratic Dijkstra}
        \label{alg:aqdijkstra}
        \begin{algorithmic}[1]
        
        \State \textbf{function} \textit{aqDikstra}($V, A, c, q, s$):
        \\
        \State Let $D_a\leftarrow +\infty$ for every $a\in A$
        \State Let $D_{(s,s)}\leftarrow 0$, $Q\leftarrow \{(s, s)\}$
        \\
            
        
        \While{$Q \neq \emptyset$}\label{alg:aqdijkstra:mainwhile}
          \State Let $a = (i, j)\in\arg\min\{D_a: a\in Q\}$\label{alg:aqdijkstra:select}
          \State Set $Q\leftarrow Q\setminus \{a\}$
          \For{each $k$ in neighbors of $j$}
              \If{$D_a + c_{jk} + q_{ijk} < D_{(j,k)}$}\label{alg:aqdijkstra:update}
                \State Set $D_{(j,k)} \gets D_a + c_{jk} + q_{ijk}$
                \State Set $Q\gets Q\cup\{(j, k)\}$
              \EndIf
            \EndFor
        \EndWhile
        \\
        \textbf{return} $D$
        \end{algorithmic}
    \end{algorithm}

The iteration over the arcs is controlled by a priority queue $Q$. Priority queues are abstract data types similar to stacks in which each element has a priority assigned to it. These queue structures are often used to speed up the search and are the standard on most implementations of Dijkstra's algorithm. The element with the highest priority is always positioned on top of the stack, forcing it to be the first to be removed. In this case, the labels with smallest accumulated costs should be prioritized since they indicate shorter paths and should be positioned on top of the stack.
At each iteration, the arc $a \in A$ with the smallest cumulative cost $D_a$ is removed from $Q$ and used as a reference to explore the neighboring nodes. Each neighbor $k$ is then evaluated and its associated cumulative cost updated according to an update criterion. The algorithm stops when $Q$ becomes empty.

The following results demonstrate the correctness of our method to compute optimal walks in the absence of negative costs and optimal acyclic paths if the graph is free of improving $\alpha$-cycles.

\begin{prop}
	If $c, q \geq 0$, the algorithm ends in a finite number of iterations and computes, for every arc $a = (i, j) \in A$, the cost of a minimum $s$-$a$-walk.
\end{prop}
\begin{proof}
	The number of iterations of the algorithm is governed by the times an arc $a\in\arg\min\{D_a: a\in Q\}$ is selected at Line \ref{alg:aqdijkstra:select}. We denote $a_t$ the arc selected at this line at iteration $t$ of the algorithm. Let us define $R_0 = \emptyset$ and for every $t \geq 1$, $R_t \leftarrow R_{t - 1}\cup \{a_t\}$. With these definitions, we have that $a_1 = (s, s), R_1 = \{(s, s)\}$.
	
	We will first demonstrate that, for every $t\geq 1$, $a_t\notin R_{t - 1}$. The first time an arc $a_t = a^*$ is added to $R_t$, it satisfies $D_{a^*}\leq D_a$ for every $a\in Q$. Since the components of $c, q$ are non-negative by assumption, it follows that any extension of any walk can only increase its cost. By the time another walk is extended using the arc $a^*$, its cost can only be $\geq D_{a^*}$. This means the cost of $D_{a^*}$ will never be updated, and therefore $a^*$ will never enter the priority queue $(Q)$ again.
	
	We will now prove that, at iteration $t$, $D_{a_t}$ is the cost of the shortest $s$-$a_t$-walk. The result is obviously true for $t = 1$. Let us assume by induction that the result holds for $t$, and let us assume by contradiction that there exists a walk $w = (w_1 = s, \ldots, w_{l - 1} = i_{t + 1}, w_l = j_{t + 1})$ such that $c(w) < D_{a_{t + 1}}$, where $a_{t + 1} = (i_{t +1}, j_{t + 1})$. Let $k$ be the first index in this walk such that $(w_k, w_{k + 1})\notin R_t$ (this is well defined since $a_{t + 1}$ satisfies this property). Let $pw_k = (w_1 = s, \ldots, w_k)$ be the $k$-th prefix of walk $w$. Since the components of $c, q$ are non-negative by assumption, it follows that $c(pw_k) + c_{w_k w_{k + 1}} + q_{w_{k - 1}w_k w_{k + 1}}\leq c(w)$. Also, because $(w_{k - 1}, w_k)\in R_t$, the induction hypothesis implies that $D_{w_{k - 1}w_k}\leq c(pw_k)$. Now, because $w_{k + 1}$ is adjacent to $(w_{k - 1}, w_k)$ in $w$, it follows from the labeling algorithm that $D_{w_k w_{k + 1}}\leq D_{w_{k - 1} w_k} + c_{w_kw_{k + 1}} + q_{w_{k - 1}w_k w_{k + 1}}$. Finally, because $a_{t +1}$ was picked at iteration $t + 1$ of the algorithm and because $(w_{k}, w_{k +1})\notin R_t$, it follows that $D_{a_{t + 1}}\leq D_{w_k w_{k + 1}}$. We put all these inequalities together to obtain:
	\begin{align*}
		D_{a_{t + 1}}&\leq D_{w_k w_{k + 1}}\\
					&\leq D_{w_{k - 1}w_k} + c_{w_k w_{k + 1}} + q_{w_{k - 1}w_k w_{k + 1}}\\
					&\leq c(pw_k) + c_{w_k w_{k + 1}} + q_{w_{k - 1}w_k w_{k + 1}}\\
					&\leq c(w)\\
					&< D_{a_{t + 1}},
	\end{align*}
	which is obviously a contradiction, proving that $D_{a_{t + 1}}$ is the minimum cost among all $s$-$a_{t + 1}$-walks.$\hfill\square$
\end{proof}

\begin{prop}
	In the absence of improving $\alpha$-cycles, $\min\{D_{a}: a\in \delta^-(t)\}$ as computed by Algorithm \ref{alg:aqdijkstra} corresponds to the minimum cost of a $s$-$t$-path in $G$.
\end{prop}
\begin{proof}
	Let us assume by contradiction that, for a given arc $a = (u, t)$, $D_a$ corresponds to the cost of a $s$-$a$-walk with a cycle and that no $s$-$a$-path exists with the same cost. This walk can be written as $(w_i)_{i = 1}^n$ with $w_1 = s,w_{n - 1} = u, w_n = t$ and $(w_i, w_{i + 1})\in A$ for every $i = 1\ldots n - 1$. If $w_l = w_{k}$ for $1 < l < k < n$, one can replace the $\alpha$-cycle $(w_i)_{i = l - 1}^{k + 1}$ by $(w_{l - 1}, w_l, w_{k + 1})$ to obtain a walk with one less cycle, at the same cost or less (because the $\alpha$-cycle is non-improving). One can repeat this argument as many times as necessary to ultimately obtain a $s$-$a$-path of the same cost or less, which contradicts our assumption. It follows that $D_a$ corresponds to the cost of a $s$-$a$-path.$\hfill\square$
\end{proof}	

These two results characterize a nice feature of our method. One does not need to incur into an expensive procedure to check for the existence of improving $\alpha$-cycles. Instead, one can check for the non-negativity of the graph weights and execute the \texttt{aqD} regardless of the existence of $\alpha$-cycles. One can \textit{a posteriori} check for the existence of improving $\alpha$-cycles once the algorithm has computed an optimal $s$-$t$-walk.

\subsection{Complexity of \texttt{aqD}}

The number of neighboring arcs that are inspected in Line \ref{alg:aqdijkstra:update} is bounded above by the number of quadratic arcs of the form $(i, j, k)$ with $\big((i, j), (j, k) \big) \in A \times A$, that we denote $|A_Q|$. Removing the arc of highest priority from $Q$ and rearranging its elements to maintain its structure can be done in $\mathcal{O}(\log|A|)$ amortized time if the priority queue is implemented using a heap. The total number of iterations (governed by the main loop in Line \ref{alg:aqdijkstra:mainwhile} of the pseudo-code) is $\mathcal{O}(|A|)$. This leads to a total complexity of $\mathcal{O}(|A|\log|A| + |A_Q|)$. For dense graphs (those for which $\mathcal{O}(|A_Q|) = \mathcal{O}(V^3)$ and $\mathcal{O}(|A|)= \mathcal{O}(V^2)$), the time complexity is of $\mathcal{O}(V^3)$. On sparse graphs (those for which $\mathcal{O}(|A|) = \mathcal{O}(|A_Q|) = \mathcal{O}(|V|)$), the time complexity is reduced to $\mathcal{O}(|V|\log|V|)$. For comparison, the standard Dijkstra's method for linear shortest paths runs in $\mathcal{O}((|A| + |V|)\log|V|)$ time using a heap, which reduces to $\mathcal{O}(|V|^2\log|V|)$ for dense graphs, and to $\mathcal{O}(|V|\log|V|)$ for sparse ones. Hence our algorithm shares the same worst-case complexity as the standard Dijkstra's method on sparse graphs, but is slower on dense graphs.


\section{A\textsuperscript{*} with backward search for cost-to-go estimation} \label{sec:astar}



In this section, we introduce a bidirectional labeling algorithm to increase the search speed. This scheme consists in solving the problem in two steps: (i) a backward search using a (linear) SPP, starting from the target node $t \in V$, to find cost-to-go estimates; (ii) a forward search for the original AQSPP, starting from the source node $s \in V$, but considering the cost-to-go estimates as lower bounds to guide the search. This approach is inspired by the work of \cite{thomas2019}, where the authors extend the seminal work of \cite{hart1968} to develop a bidirectional search for resource-constrained shortest path problems. The idea is built upon Hart et al.'s demonstration that A\textsuperscript{*} only terminates with an optimal path when the estimates of the cost-to-go are actual lower bounds. In the context of this work, for the case of a quadratic shortest path search, linear estimators of the costs-to-go are actual lower bounds for the quadratic problem provided that the quadratic terms are reduced --- or as in this case, ignored. The fundamental idea is that these linear estimations can be calculated through a backward search from the target node $t \in V$ to all other nodes by solving the standard (linear) SPP.


Structurally, the algorithm starts with a backward search for the linear SPP where the quadratic costs are neglected. Starting from the sink node $t \in V$, one finds, for every node $u \in V$, the shortest distance from $u$ to $t$ using the linear costs $\boldsymbol{c}$. We denote $(B_u)_{u \in V}$ the costs computed using this procedure. Note that the algorithm not only is fast even for dense graphs (it relies on solving the standard SPP using Dijkstra's method) but also occupies only linear space to store the information. Note that this method also provides an upper bound for the AQSPP by evaluating the $s$-$t$-path solution provided by considering the quadratic costs.


In the second step, we perform a forward search using the proposed \texttt{aqD} method but taking into account the cost-to-go estimators to guide the search. More specifically, for every arc $a = (i, j) \in A$, $D_a + B_j$ represents a lower bound on the cost that would be achieved by extending the current $s$-$a$-walk to $t \in V$. These values are then used to select the next arc to process in the \texttt{aqD} method. By estimating the cost of arriving at $t \in V$, we can significantly reduce the size of the search tree and speed up the procedure. However, it is important to be aware that these estimates must always correspond to lower bounds on the costs-to-go to guarantee the correctness of the method. 

The proposed method is presented on Algorithm \ref{alg:bidirectional}.

\begin{spacing}{1.5}
    \begin{algorithm}[htbp]
        \caption{Adjacent Quadratic A\textsuperscript{*} with backward cost-to-go estimation}
        \label{alg:bidirectional}
        \begin{algorithmic}[1]
        
        \State \textbf{function} \textit{aqA\textsuperscript{*}}($G$,$c$,$q$,$s$,$t$):
        \\
        \State Let $(B_u)_{u\in V} \gets Dijkstra(G,c, t)$
        \State Let $D_a\leftarrow +\infty$ for every $a\in A$
        \State Let $D_{(s,s)}\leftarrow 0$, $Q\leftarrow \{(s, s)\}$
        \\
        
        \While{$Q \neq \emptyset$}\label{alg:aqastar:mainwhile}
        \State Let $a = (i, j)\in\arg\min\{D_a + B_j: a = (i, j)\in Q\}$\label{alg:aqaqastar:select}
        \State Set $Q\leftarrow Q\setminus \{a\}$
        \For{each $k$ in neighbors of $j$}
        \If{$D_a + c_{jk} + q_{ijk} < D_{(j,k)}$}\label{alg:aqaqastar:update}
        \State Set $D_{(j,k)} \gets D_a + c_{jk} + q_{ijk}$
        \State Set $Q\gets Q\cup\{(j, k)\}$
        \EndIf
        \EndFor
        \EndWhile
        \\
        \textbf{return} $D$

        \end{algorithmic}
    \end{algorithm}
\end{spacing}

\section{Solving very-large scale AQSPPs in practice}
\label{sec:on-the-fly}

On both proposed algorithms --- \texttt{aqD} and \texttt{aqA\textsuperscript{*}} ---, the adjacent quadratic cost matrix $\boldsymbol{q}$ are assumed to be pre-calculated and allocated on some existing data structure. Although having such information allocated in memory improves the search speed of the algorithm, it can become impractical as the problems increase in size. For instance, when dealing with very-large-scale road networks, the graphs can easily achieve dimensions of millions of nodes, tens of millions of arcs, and hundreds of millions of quadratic arcs. Thus, storing this information in memory becomes impractical.

In this section, we propose a follow-on improvement to \texttt{aqD} and \texttt{aqA\textsuperscript{*}}, which embeds the assessment of the adjacent quadratic costs within the algorithms (on the fly). This is possible whenever there exists a function $\Gamma : V^3 \rightarrow R^+$ able to compute quadratic costs for every three indices $(i,j,k)$. We will refer to these modified algorithms as \texttt{aqD+} and \texttt{aqA\textsuperscript{*}+}.

Several applications could benefit from this improvement and could build a function like $\Gamma$. The problem of spatially defining a route, for example, often requires path search on massive instances. Those problems usually use three-dimensional spatial data, which can be converted into extensive graphs used to define the routes. Additionally, attributes such as slope, angle, and others can be used as input for $\Gamma$ due to their psychical nature. Large infrastructure projects, such as highways, electric transmission lines, and gas pipelines rely on this framework.


\section{Computational Experiments} \label{sec:experiments}

In this section, we focus on evaluating the performance of the proposed algorithms, presenting results from a set of computational experiments. We explore random graph instances generated with two different methods and observe their effects on the search behavior. Then, we analyze the algorithms' sensitivity to variations of the quadratic cost. On the third test, we compare the algorithms' performance on real grid instances with different sizes. Finally, the last experiment perform a stress test, where we increase instances sizes up to $16 \times 10^6$ nodes, $64 \times 10^6$ arcs, and $10^9$ quadratic arcs.

The computational analyses evaluate the total search time required by the proposed algorithms against the method presented in \cite{rostami2015} and reviewed in \cite{hu2018} that solves AQSPPs for directed acyclic graphs. This algorithm re-writes the original graph according to a linearization of adjacent quadratic costs. Thus, the optimal path can be found on an extended graph through standard SPP algorithms. We will refer to this benchmark as \texttt{Lin} in the following sections. The results from solving the BQP using off-the-shelf solvers are not presented due to the solver's inability to handle even the smallest instance tested in our experiments.

\subsection{Software and Machine Settings}

The Julia programming language is used throughout this work. The package \texttt{Graphs.jl} \citep{graphs-2021} is chosen to handle graph data structures, as it provides efficient data structures to handle graph objects and mature methods to solve shortest path problems. It should be highlighted that our implementation of the \texttt{aqD} algorithm is made on top of Dijkstra's algorithm implementation provided by this package.

The R programming language was also used in this paper, mainly for Geographic Information System (GIS) purposes. The packages \texttt{rgdal} and \texttt{rgeos} were used for manipulation of geographical data that was further converted into the graphs used for the computational experiments and the application study. All source code used in this work is publicly available for further verification, validation, and improvements\footnote{The code will be made publicly available as a GitHub repository upon acceptance}. Finally, computational calculations were made on an Intel\textsuperscript{\textregistered} Core i7-10700K 4.8GHz with 64 GB of RAM machine, using one single thread.

\subsection{Random graphs}

In the first experiment, we study the algorithms' performance considering randomly generated graph instances. We aim to analyze with this experiment how does random graphs topology impacts the search process. The methods used to build the instances are known as Erdős-Rényi \citep{erdos-renyi1959,Bollobas2001} and Configuration Model \citep{newman2010}, which have been widely used in scientific literature. More specifically, the Erdős-Rényi methods defines a graph where the arcs are added between pairs of nodes according to a given probability $p \in [0,1]$. The greater $p$ is, greater is the number of arcs, thus the denser the graph is. For expository purposes, in this model, we define $p$ as $0.8$ to obtain highly connected instances and range the number of nodes $|V| \in \{100,150,200,250,300,350,400\}$. For each configuration $(|V|,p)$, 100 random instances are generated. Costs for arcs and quadratic arcs are also randomly generated using a standard Uniform distribution, i.e., $\big(q(i,j,k), c\big((i,j)\big)\big) \sim Unif(0,1)$.

Figure \ref{fig:random-erdos} presents compiled results from this experiment, showcasing the minimum, average, and maximum search times for each algorithm and size tested. One can notice a clear variation in search time for each of the tested algorithms. Results for \texttt{aqA\textsuperscript{*}} and \texttt{aqD} show reduced and similar search-times throughout the instances. This behavior can be explained by the high connectivity of the Erdős-Rényi instances considered in this experiment. Since $p$ was set to $0.8$, there is a high probability that the source node is close-connected to the target node, thus the shortest path contains only a few nodes. Consequently, the number of nodes needed to be visited becomes small, shortening the search process.

\begin{figure}[htbp]
	\centering
	\includegraphics[width=\textwidth]{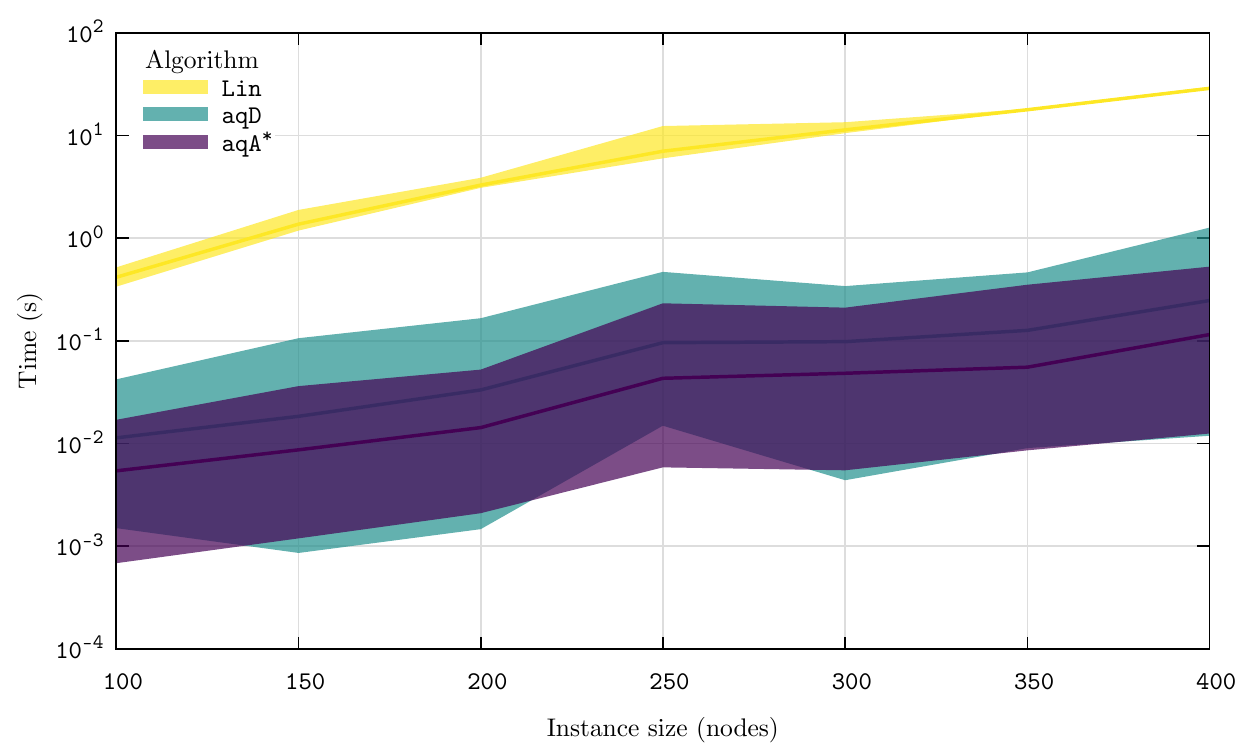}
	\caption{Results for Erdős-Rényi instances}
	\label{fig:random-erdos}
	\hfill
\end{figure}

Furthermore, we highlight the difference in the search-time magnitude for the benchmark \texttt{Lin} approach, in comparison to \texttt{aqA\textsuperscript{*}} and \texttt{aqD}. This difference is primarily due to the time spent building the linearized graph. Since the number of quadratic arcs is large, performing the linearization is very time-consuming and becomes a drawback to this approach.

The Configuration Model generates a random graph $G$ from a given degree sequence $\boldsymbol{\kappa}=\{\kappa_{i}\}_{i = 1}^{|V|}$. A degree defines the number of connections a given node has, therefore, as $\boldsymbol{\kappa}$ increases, so does the number of arcs in $G$. To build the graph, the method randomly samples $\kappa_{i}$ nodes from $G$ and connects them to node $i \in V$. For this model, we fix the nodes' degrees to $8$ and vary the number of nodes $|V| \in \{150\times10^3,200\times10^3,250\times10^3,300\times10^3\}$. Similarly, 100 random instances are generated with $c\big((i,j)\big)$ and $q(i,j,k)$ drawn from an $Unif(0,1)$. Results from the Configuration Model, presented in Figure \ref{fig:random-conf-model}, suggest that the speed ups remain nearly constant, with our method \texttt{aqA\textsuperscript{*}} performing around 32 times faster on average than \texttt{Lin}. Additionally, \texttt{aqD} is about 5 times faster than \texttt{Lin}.

\begin{figure}[htbp]
	\centering
	\includegraphics[width=\textwidth]{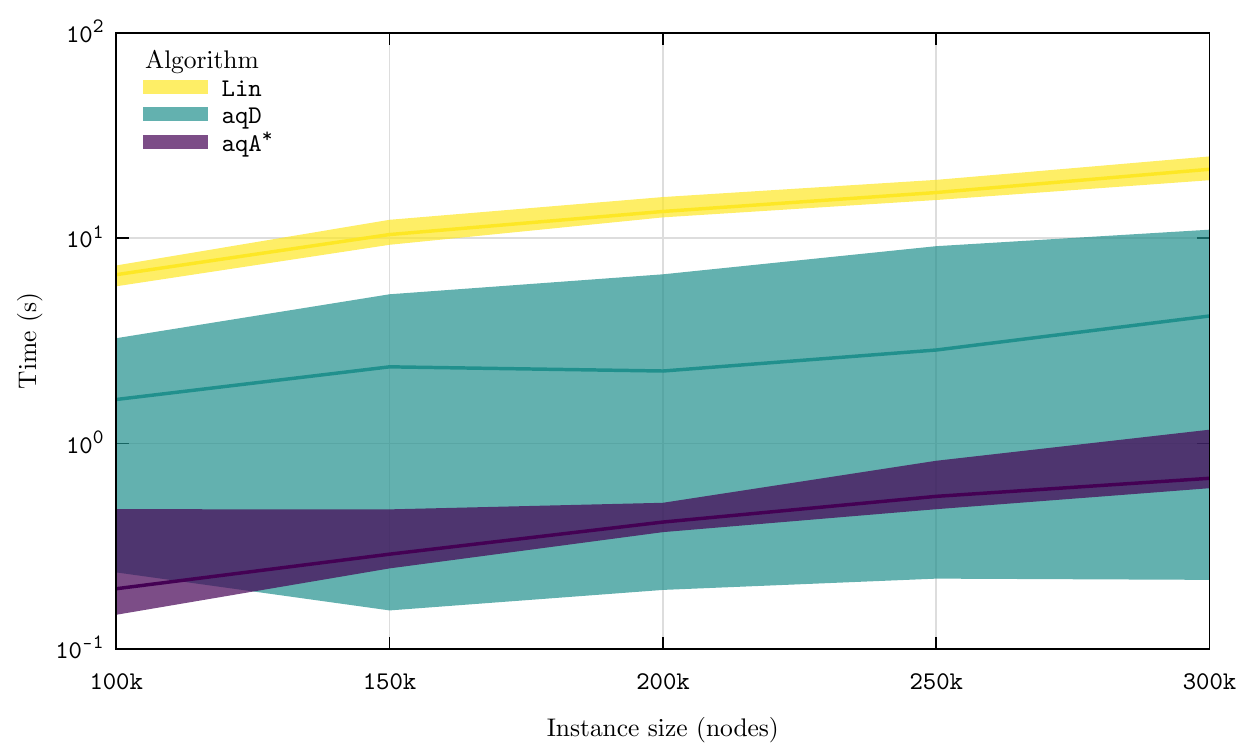}
	\caption{Results for Configuration Model instances}
	\label{fig:random-conf-model}
	\hfill
\end{figure}

Furthermore, the presented results indicate that the \texttt{aqA\textsuperscript{*}} exhibits less variance than \texttt{aqD}, resulting in a more consistent and predictable search time for instances generated with the Configuration Model. These results reaffirm the role of the backward step on providing reasonable cost-to-go estimates and consistently speeding up the search process. 

\subsection{Quadratic Cost Variation}

The second experiment measures the sensitivity of the algorithms to the magnitude of quadratic costs. A parameter $\lambda \in \{1, 5, 50, 500, 1000\}$ is set to scale the quadratic costs, so that the arc interaction would be $\lambda q(i,j,k)$. We test $100$ instances generated with the Erdős-Rényi method with $400$ nodes and $100$ instances generated with the Configuration Model with $300 \times 10^3$ nodes for each $\lambda$ value. Figures \ref{fig:lambda-erdos} and \ref{fig:lambda-conf-model} present total search time for both graphs, and for each $\lambda$ value considered. The presented results suggest that \texttt{aqA\textsuperscript{*}}, \texttt{aqD} and \texttt{Lin} are not significantly influenced by an increase of quadratic costs. We also do not observe any apparent trends or unconventional behaviour on the plots, confirming the insensitivity towards quadratic cost variation.

\begin{figure}[htbp]
	\centering
	\includegraphics[width=\textwidth]{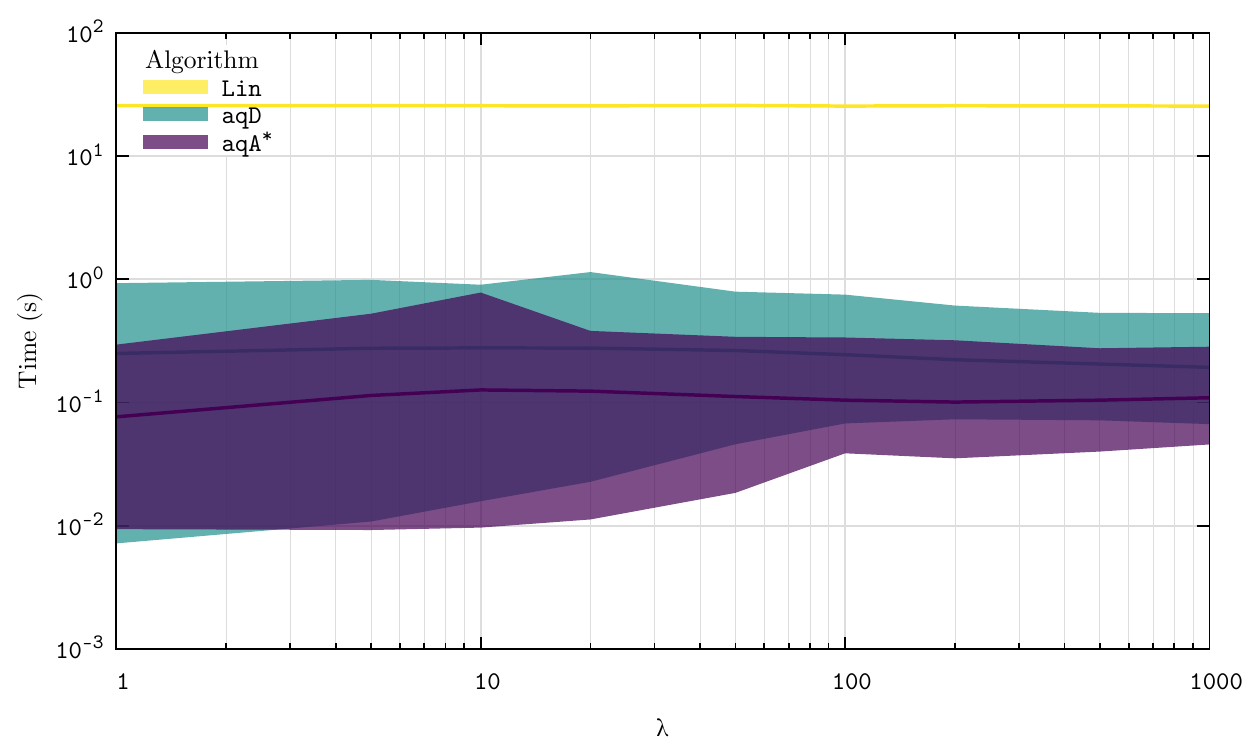}
	\caption{Results for Erdős-Rényi instances with $\lambda$ variation}
	\label{fig:lambda-erdos}
	\hfill
\end{figure}

\begin{figure}[htbp]
	\centering
	\includegraphics[width=\textwidth]{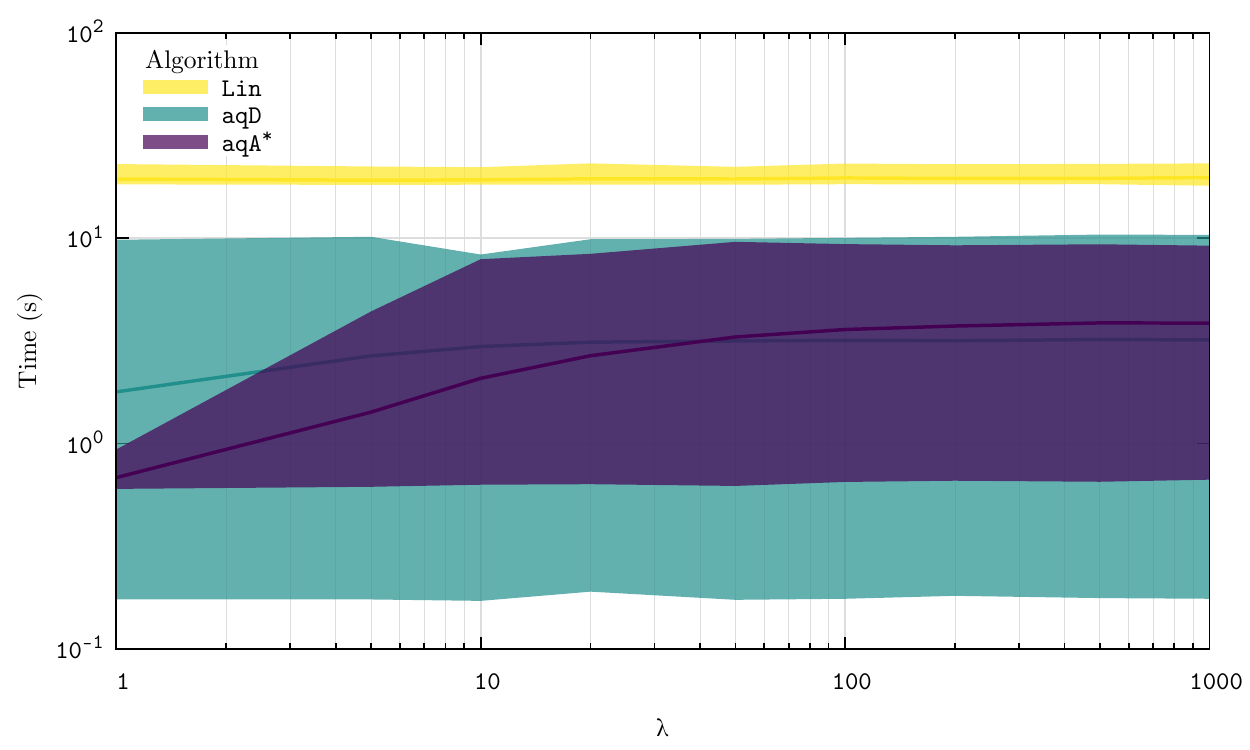}
	\caption{Results for Configuration Model instances with $\lambda$ variation}
	\label{fig:lambda-conf-model}
	\hfill
\end{figure}

These results highlight the solution capability of \texttt{aqA\textsuperscript{*}} since it uses a cost-to-go estimator that we would expect to lose relevance as quadratic costs increase. Nevertheless, the $\lambda$ scale factors chosen for these experiments do not reverse the costs-to-go into a poor estimator. In other words, even with quadratic costs scaled up by a factor of $\lambda = 1,000$, the estimator still provides valuable guidance for the search process. This behavior highlights the robustness of the proposed algorithm.

\subsection{Grid Graphs}
\label{sec:grid-graph}

For the following experiments, we consider actual grid graph instances of spatial nature. Each node is associated with an actual location on the globe, and the arc costs represent the distances between locations. The graphs are built out of a regularly spaced grid of elevation points, which we define on \eqref{eq:y-matrix} as the matrix $Y$ with $N$ rows and $M$ columns. Each element of the matrix is associated with a numeric value and spatial coordinates.

\begin{equation}
    \centering
        Y=\begin{bmatrix}
        y_{11} & \dots  & y_{1M} \\
        \vdots & \ddots & \vdots \\
        y_{N1} & \dots  & y_{NM} \\
    \end{bmatrix}
    \label{eq:y-matrix}
\end{equation}

To create a graph instance, we convert the matrix $Y$ into a weighted graph, where each node $i \in V$ is equivalent to the element $y_{nm}$ with $|V|$ the product of matrix dimensions $N$×$M$, and $|A|$ can range according to a graph composition rule. This rule is a function of how many neighbors $n$ each node connects to. Figure \ref{fig:node-neighborhood} presents the effect of choosing different number of neighbors on $|A|$.

\begin{figure}[htbp]
	\centering
	\includegraphics[width=0.5\textwidth]{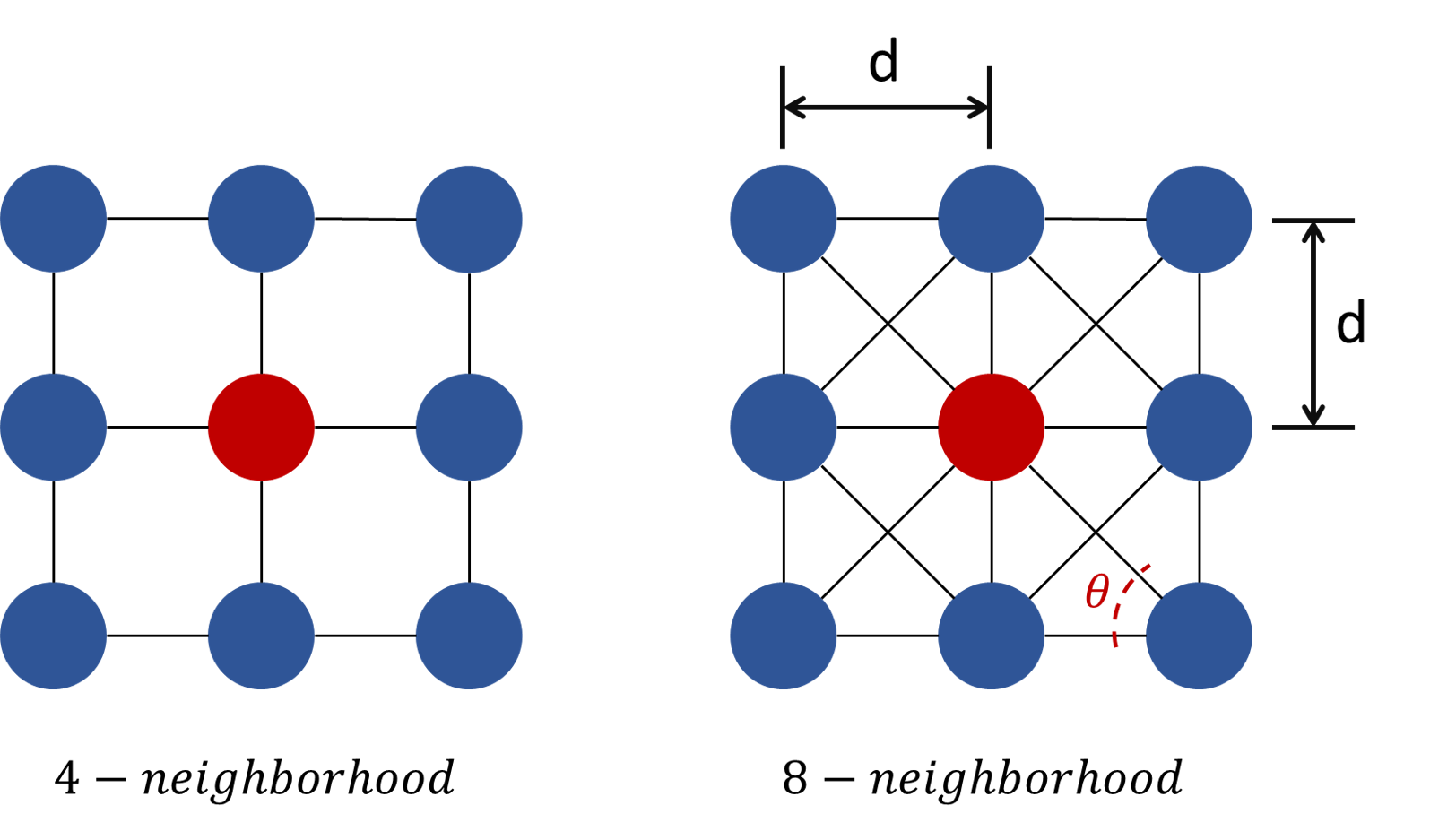}
	\caption{Illustration of node neighborhood for $n=8$ and $n=16$}
	\label{fig:node-neighborhood}
\end{figure}

The intuition behind node neighborhood comes from direction and length of movement. Common choices for $n$ are $\{2,4,8\}$. For instance, choosing $n = 2$ or $n = 4$ represent unitary movements on two directions (east and south) or on four directions (west/east and north/south), respectively. On the other hand, for $n = 8$, diagonal movements are contemplated.

We denote by $d(y_{nm},y_{op})$ as a function that calculates the distance cost between points $y_{nm}$ and $y_{op}$. Thus, arc linear costs are defined as $\{c\big((i,j)\big) = d(y_{nm},y_{op}) \quad | \quad i = (n,m) \quad \text{and} \quad j=(o,p)\}$, where $(n,m)$ and $(o,p)$ are neighbors according to a given graph composition rule. For expository purposes, in this work, we set $n = 8$ so that diagonal movements are allowed. As we are working with spatial information, diagonal costs must be adjusted by a factor of $\sqrt{2}$ following standard trigonometric developments. Figure \ref{fig:benchmark-tests-1} presents the results for all considered algorithms.

\begin{figure}[htbp]
	\centering
	\includegraphics[width=\textwidth]{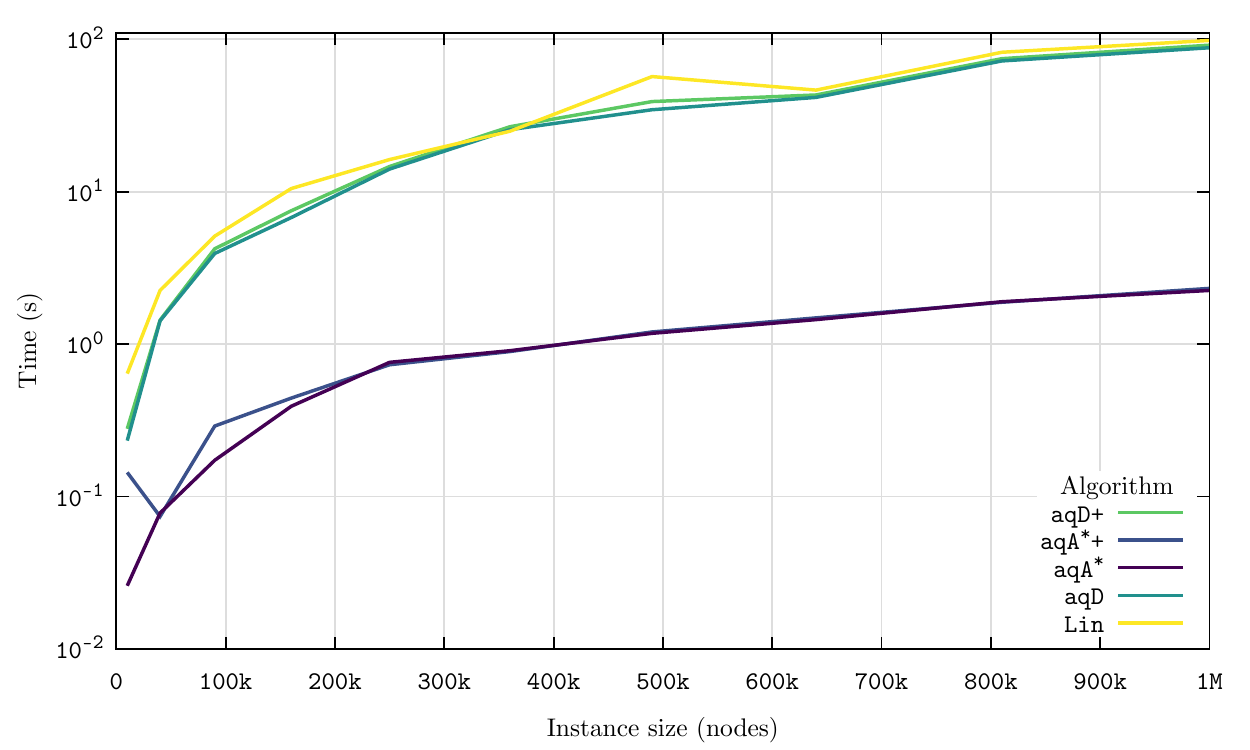}
	\caption{Grid Graphs: Search Time Comparison}
	\label{fig:benchmark-tests-1}
\end{figure}

Numerical results depicted in Figure \ref{fig:benchmark-tests-1} present the \texttt{aqA\textsuperscript{*}} as the fastest algorithm among all AQSPP algorithms considered. The \texttt{aqA\textsuperscript{*}} algorithm is approximately 39 times faster than \texttt{aqD} and roughly 53 times faster than \texttt{Lin} approach on average. Also, it seems the ratios fluctuate around this average value, which suggest a constant relation. A possible explanation for this over-performance is related to the similarity of paths obtained from SPP and AQSPP algorithms, resulting in exact cost-to-go estimators, thus speeding up \texttt{aqA\textsuperscript{*}}. \texttt{Lin}'s method, on the other hand, is outperformed by all the algorithms proposed in this work. The results displayed in Figure \ref{fig:benchmark-tests-1} show an \texttt{aqD} improvement in speed ranging between 10\% up to a 50\% in comparison to the \texttt{Lin} approach. Further analyses revealed that the time spent to linearize the graph is very significant and becomes a downside to the \texttt{Lin} approach, especially as instances grow in size. Figure \ref{fig:benchmark-tests-1} helps us notice the similarity between the original and modified versions of our algorithms, for both \texttt{aqD} and \texttt{aqA\textsuperscript{*}}. Although calculating the quadratic costs on the run lowers the search speed, the impact is negligible. Since the modified versions benefit from memory allocation reductions and allow solving larger instances, we highly recommend their use for applications where adaptation is a viable resource. 

\subsection{Stress test}

We also provide a final experiment designed to stress search capabilities on extremely large-scale instances, ranging from $144\times10^6$ to $10^9$ quadratic arcs. Since real applications may require path search on these instance sizes, evaluating the algorithms performance over increasingly larger graph instances can provide meaningful insights. Figure \ref{fig:stress} displays the results from these experiments. It presents the search time for all approaches considered in this work using the extremely large-scale grid graph instances, generated following the same strategy as in Section \ref{sec:grid-graph}.

\begin{figure}[htbp]
	\centering
	\includegraphics[width=\textwidth]{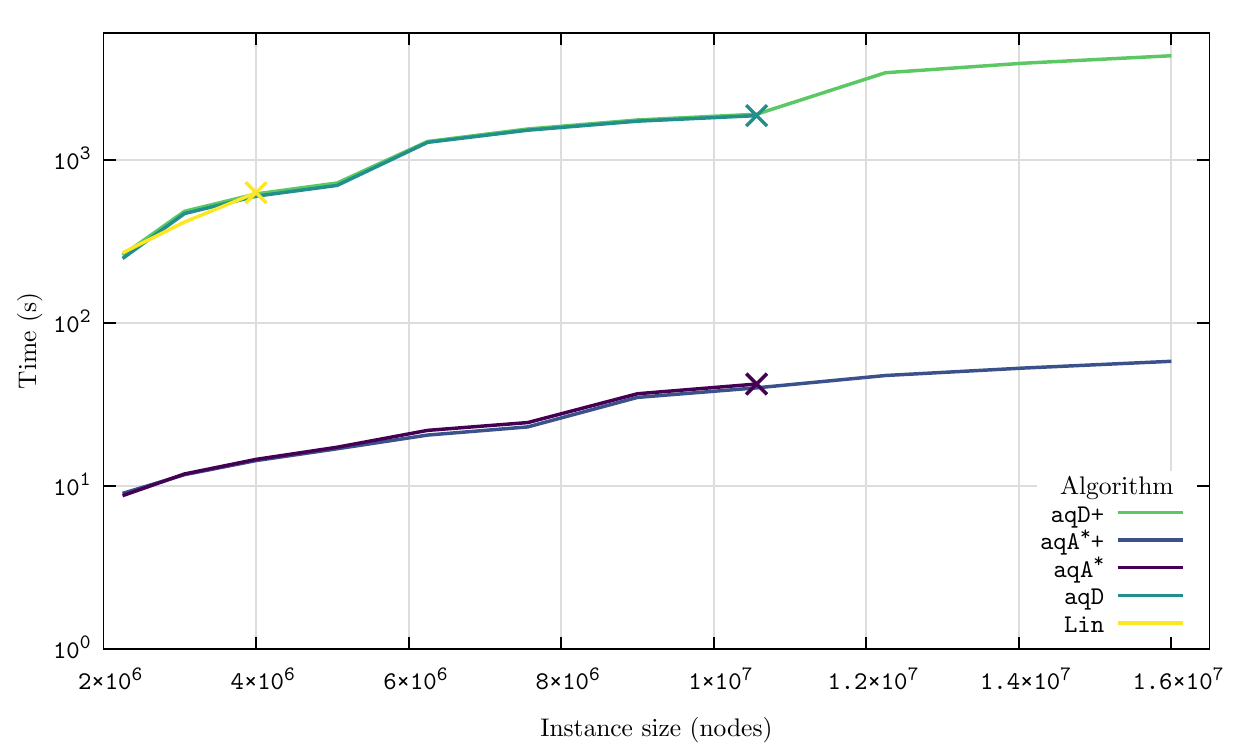}
	\caption{Stress Analysis: Search-Time Comparison}
	\label{fig:stress}
	\hfill
\end{figure}

The results show that none of the original algorithms is able to solve all instances due to memory issues. Since the methods assume a preexisting structure that maps adjacent arcs into quadratic costs, storing this data become challenging as instances grow. Furthermore, \texttt{aqD} and \texttt{aqA\textsuperscript{*}} have internal data structures to store relevant information throughout the search process, which also requires memory space. The \texttt{Lin} approach presented an even higher demand for memory allocation since, by construction, there is a need to store the linearized graph.
The marker on each approach in Figure \ref{fig:stress} displays the maximum instance size each approach is able to solve. The large instance \texttt{Lin} is able to solve has $256 \times 10^6$ quadratic nodes. On the other hand, \texttt{aqD} and \texttt{aqA\textsuperscript{*}} algorithms have similar extensions, solving instances up to $676\times10^6$ quadratic nodes, but \texttt{aqA\textsuperscript{*}} algorithm obtains the same solution in $2.26\%$ of the \texttt{aqD} computing time.
In contrast, the modified versions of \texttt{aqD} and \texttt{aqA\textsuperscript{*}} are able to solve all instances tested. Since quadratic calculations are made throughout the search, these algorithms do not require significant memory space, allowing them to handle increasingly larger instances properly. As expected, \texttt{aqA\textsuperscript{*}+} remains the faster approach, solving the AQSPP in roughly 60.0 seconds on a graph with $16\times10^6$ nodes, $64\times10^6$ arcs and $10^9$ quadratic arcs, $1.34\%$ of the \texttt{aqD+} computing time.

\section{Conclusion} \label{sec:conclusion}

This study provides a theoretical description of Adjacent Quadratic Shortest Problems and proposes an extension of Dijkstra's algorithm (\texttt{aqD}) to solve AQSPPs on graphs with no improving $\alpha$-cycles. Two enhancement procedures for the primary approach are also presented in this work: an adjacent quadratic version of A\textsuperscript{*} with backward cost-to-go estimation and an embedded calculation of adjacent quadratic costs. 
From the theoretical perspective, we present evidence that \texttt{aqD} can solve AQSPPs in polynomial time to proven optimality. We also show through computational experiments that \texttt{aqD} outperforms the algorithm proposed by \cite{rostami2015}, being more than 75$\times$ faster on very-large scale instances and a high scalability capability. This behavior is also valid for the proposed algorithm \texttt{aqA\textsuperscript{*}}, which was the fastest approach among the ones studied, outperforming \texttt{aqD} search speed by about the same amount on the benchmark analysis.
Furthermore, other test results suggest that the proposed algorithms also had significant performance under randomly generated instances, indicating their robustness. In addition, we identified a significant drawback on \texttt{Lin} method \citep{rostami2015} for highly connected graphs related to the time spent performing the linearization step.
Moreover, the algorithms were also indifferent to adjacent quadratic cost variation, already expected for \texttt{aqD} and an interesting result for \texttt{aqA\textsuperscript{*}}. Finally, results suggest that even for a high quadratic-linear cost ratio, the backward search is still profoundly relevant to enhance the \texttt{aqA\textsuperscript{*}} solution capability.

Future research shall focus on extending our findings to elementary variants of the AQSPP, relevant when the graph contains negative ($\alpha$-)cycles. This would be critical in quadratic variants of vehicle routing problems \citep{fischer-2013,martinelli2015} when solved by column generation. Another interesting avenue for future research would be to extend our approach to the case where no negative ($\alpha$-)cycles exist, but where some (pairs of) arcs may be associated to negative costs. In this context, one shall develop label-correcting variants of our method in the spirit of \cite{bellman1958}.

\section*{Acknowledgments}

This study was partially supported by PUC-Rio, by the Coordenação de Aperfeiçoamento de Pessoal de Nível Superior – Brasil (CAPES) – Finance Code 001, by the Conselho Nacional de Desenvolvimento Científico e Tecnológico (CNPq) under grant numbers 315361/2020-4 and 422470/2021-0, by Fundação (FAPERJ) under grant numbers E-26/010.101261/2018, E-26/202.825/2019, E-26/010.002282/2019 and E-26/010.002232/2019, and by the Natural Sciences and Engineering Research Council of Canada (NSERC) under Grant number 2020-06311.

\bibliographystyle{elsarticle-harv}
{\linespread{1.3}\selectfont\bibliography{aqdijkstra}}

\input{appendix}

\end{document}

%% file: appendix.tex
\appendix

\section{Detailed Computational Results}

In this section, we present the complete computational results obtained in the experiments of this paper.

\subsection{Random Graphs}

In Tables \ref{tbl:erdos} and \ref{tbl:conf-model}, we present the detailed results for the Erdős-Rényi and the Configuration Model random graphs. Both generation algorithms are tested for different random graph sizes, 100 graphs for each size. For each size and graph, we present the search time of \texttt{Lin}, \texttt{aqD}, and \texttt{aqA\textsuperscript{*}} approaches.

\subsection{Quadratic Cost Variation}

In Tables \ref{tbl:lambda-erdos} and \ref{tbl:lambda-conf-model}, we present the detailed results for the quadratic cost variation experiment using the Erdős-Rényi and the Configuration Model random graphs. Both generation algorithms are tested for different $\lambda$ values, 100 graphs for each value. For each value and graph, we present the search time of \texttt{Lin}, \texttt{aqD}, and \texttt{aqA\textsuperscript{*}} approaches.

\subsection{Grid Graphs}

Table \ref{tab:grid-graphs} presents the detailed results for the grid graph instances. For each instance size, we present the search time for all tested approaches, including the modified versions of \texttt{aqD} and \texttt{aqA\textsuperscript{*}}.

\subsection{Stress Test}

Table \ref{tab:stress} presents the detailed results for the stress test using the grid graph instances. For each instance size, we present the search time for all tested approaches, including the modified versions of \texttt{aqD} and \texttt{aqA\textsuperscript{*}}.

\begin{landscape}
\setlength{\tabcolsep}{2.7pt}
\begin{longtable}{crrrrrrrrrrrrrrrrrrrrrrrrrrr}\hline
\multirow{2}{*}{Inst} & \multicolumn{3}{c}{100 nodes} && \multicolumn{3}{c}{150 nodes} && \multicolumn{3}{c}{200 nodes} && \multicolumn{3}{c}{250 nodes} && \multicolumn{3}{c}{300 nodes} && \multicolumn{3}{c}{350 nodes} && \multicolumn{3}{c}{400 nodes} \\\cline{2-4}\cline{6-8}\cline{10-12}\cline{14-16}\cline{18-20}\cline{22-24}\cline{26-28}
& \multicolumn{1}{c}{\texttt{Lin}} & \multicolumn{1}{c}{\texttt{aqD}} & aqA* && \multicolumn{1}{c}{\texttt{Lin}} & \multicolumn{1}{c}{\texttt{aqD}} & \multicolumn{1}{c}{\texttt{aqA\textsuperscript{*}}} && \multicolumn{1}{c}{\texttt{Lin}} & \multicolumn{1}{c}{\texttt{aqD}} & \multicolumn{1}{c}{\texttt{aqA\textsuperscript{*}}} && \multicolumn{1}{c}{\texttt{Lin}} & \multicolumn{1}{c}{\texttt{aqD}} & \multicolumn{1}{c}{\texttt{aqA\textsuperscript{*}}} && \multicolumn{1}{c}{\texttt{Lin}} & \multicolumn{1}{c}{\texttt{aqD}} & \multicolumn{1}{c}{\texttt{aqA\textsuperscript{*}}} && \multicolumn{1}{c}{\texttt{Lin}} & \multicolumn{1}{c}{\texttt{aqD}} & \multicolumn{1}{c}{\texttt{aqA\textsuperscript{*}}} && \multicolumn{1}{c}{\texttt{Lin}} & \multicolumn{1}{c}{\texttt{aqD}} & \multicolumn{1}{c}{\texttt{aqA\textsuperscript{*}}} \\\hline
\endfirsthead
\hline
\endhead
\hline
\multicolumn{28}{r}{{Continued on next page}}
\endfoot
\hline\hline
\caption{Results for Erdős-Rényi random graphs}\label{tbl:erdos}
\endlastfoot
1 & 0.358 & 0.009 & 0.006 && 1.328 & 0.061 & 0.004 && 3.617 & 0.006 & 0.017 && 7.303 & 0.027 & 0.009 && 12.362 & 0.183 & 0.028 && 17.787 & 0.155 & 0.015 && 28.916 & 0.216 & 0.072 \\
2 & 0.423 & 0.006 & 0.007 && 1.192 & 0.014 & 0.021 && 3.463 & 0.009 & 0.011 && 6.903 & 0.096 & 0.043 && 12.424 & 0.042 & 0.012 && 17.772 & 0.093 & 0.094 && 29.024 & 0.515 & 0.119 \\
3 & 0.492 & 0.006 & 0.004 && 1.528 & 0.012 & 0.004 && 3.280 & 0.038 & 0.003 && 6.411 & 0.049 & 0.014 && 11.369 & 0.048 & 0.114 && 17.828 & 0.183 & 0.057 && 28.715 & 0.019 & 0.055 \\
4 & 0.399 & 0.011 & 0.001 && 1.353 & 0.031 & 0.007 && 3.191 & 0.008 & 0.017 && 6.746 & 0.117 & 0.070 && 11.325 & 0.076 & 0.047 && 17.968 & 0.197 & 0.028 && 29.059 & 0.481 & 0.114 \\
5 & 0.427 & 0.009 & 0.008 && 1.380 & 0.013 & 0.008 && 3.282 & 0.050 & 0.003 && 6.985 & 0.102 & 0.065 && 11.126 & 0.293 & 0.026 && 17.683 & 0.206 & 0.032 && 28.808 & 0.197 & 0.530 \\
6 & 0.392 & 0.026 & 0.007 && 1.366 & 0.036 & 0.002 && 3.393 & 0.046 & 0.031 && 6.577 & 0.135 & 0.066 && 11.156 & 0.052 & 0.037 && 17.822 & 0.074 & 0.016 && 29.351 & 0.596 & 0.016 \\
7 & 0.371 & 0.016 & 0.006 && 1.369 & 0.012 & 0.008 && 3.295 & 0.027 & 0.022 && 6.490 & 0.088 & 0.012 && 11.327 & 0.075 & 0.030 && 17.699 & 0.060 & 0.058 && 29.396 & 0.073 & 0.028 \\
8 & 0.414 & 0.007 & 0.005 && 1.356 & 0.013 & 0.004 && 3.195 & 0.004 & 0.003 && 6.370 & 0.028 & 0.061 && 11.355 & 0.209 & 0.110 && 17.933 & 0.072 & 0.022 && 28.795 & 0.392 & 0.398 \\
\end{longtable}

\setlength{\tabcolsep}{4.5pt}
\begin{longtable}{crrrrrrrrrrrrrrrrrrr}\hline
\multirow{2}{*}{Inst} & \multicolumn{3}{c}{100k nodes} && \multicolumn{3}{c}{150k nodes} && \multicolumn{3}{c}{200k nodes} && \multicolumn{3}{c}{250k nodes} && \multicolumn{3}{c}{300k nodes} \\\cline{2-4}\cline{6-8}\cline{10-12}\cline{14-16}\cline{18-20}
& \multicolumn{1}{c}{\texttt{Lin}} & \multicolumn{1}{c}{\texttt{aqD}} & aqA* && \multicolumn{1}{c}{\texttt{Lin}} & \multicolumn{1}{c}{\texttt{aqD}} & \multicolumn{1}{c}{\texttt{aqA\textsuperscript{*}}} && \multicolumn{1}{c}{\texttt{Lin}} & \multicolumn{1}{c}{\texttt{aqD}} & \multicolumn{1}{c}{\texttt{aqA\textsuperscript{*}}} && \multicolumn{1}{c}{\texttt{Lin}} & \multicolumn{1}{c}{\texttt{aqD}} & \multicolumn{1}{c}{\texttt{aqA\textsuperscript{*}}} && \multicolumn{1}{c}{\texttt{Lin}} & \multicolumn{1}{c}{\texttt{aqD}} & \multicolumn{1}{c}{\texttt{aqA\textsuperscript{*}}} \\\hline
\endfirsthead
\hline
\endhead
\hline
\multicolumn{20}{r}{{Continued on next page}}
\endfoot
\hline\hline
\caption{Results for Configuration Model random graphs}\label{tbl:conf-model}
\endlastfoot
1 & 6.100 & 2.616 & 0.204 && 10.843 & 3.267 & 0.298 && 12.860 & 5.341 & 0.467 && 15.809 & 6.519 & 0.502 && 23.250 & 1.590 & 0.684 \\
2 & 6.185 & 3.263 & 0.196 && 10.870 & 1.599 & 0.264 && 14.650 & 1.054 & 0.426 && 17.455 & 2.198 & 0.510 && 21.749 & 5.020 & 0.641 \\
3 & 6.593 & 0.760 & 0.207 && 10.867 & 2.189 & 0.270 && 12.942 & 1.150 & 0.388 && 17.602 & 0.383 & 0.551 && 19.976 & 5.823 & 0.734 \\
4 & 6.852 & 2.607 & 0.181 && 10.638 & 3.254 & 0.316 && 13.344 & 0.854 & 0.406 && 16.182 & 1.159 & 0.537 && 21.656 & 0.828 & 0.629 \\
5 & 6.501 & 0.660 & 0.247 && 12.308 & 0.885 & 0.262 && 14.079 & 1.721 & 0.449 && 16.111 & 6.560 & 0.504 && 24.461 & 1.194 & 0.673 \\
6 & 6.619 & 0.294 & 0.172 && 11.375 & 1.280 & 0.266 && 12.710 & 5.231 & 0.441 && 16.725 & 0.884 & 0.530 && 22.030 & 1.588 & 0.635 \\
7 & 6.658 & 2.569 & 0.221 && 11.712 & 3.994 & 0.257 && 12.985 & 2.069 & 0.426 && 16.803 & 1.680 & 0.592 && 22.364 & 3.337 & 0.627 \\
8 & 7.000 & 2.955 & 0.157 && 11.590 & 1.900 & 0.358 && 13.603 & 4.593 & 0.394 && 15.565 & 3.262 & 0.514 && 23.157 & 7.552 & 0.613 \\
\end{longtable}

\setlength{\tabcolsep}{4.5pt}
\begin{longtable}{crrrrrrrrrrrrrrrrrrr}\hline
\multirow{2}{*}{Inst} & \multicolumn{3}{c}{$\lambda = 1$} && \multicolumn{3}{c}{$\lambda = 5$} && \multicolumn{3}{c}{$\lambda = 50$} && \multicolumn{3}{c}{$\lambda = 500$} && \multicolumn{3}{c}{$\lambda = 1000$} \\\cline{2-4}\cline{6-8}\cline{10-12}\cline{14-16}\cline{18-20}
& \multicolumn{1}{c}{\texttt{Lin}} & \multicolumn{1}{c}{\texttt{aqD}} & aqA* && \multicolumn{1}{c}{\texttt{Lin}} & \multicolumn{1}{c}{\texttt{aqD}} & \multicolumn{1}{c}{\texttt{aqA\textsuperscript{*}}} && \multicolumn{1}{c}{\texttt{Lin}} & \multicolumn{1}{c}{\texttt{aqD}} & \multicolumn{1}{c}{\texttt{aqA\textsuperscript{*}}} && \multicolumn{1}{c}{\texttt{Lin}} & \multicolumn{1}{c}{\texttt{aqD}} & \multicolumn{1}{c}{\texttt{aqA\textsuperscript{*}}} && \multicolumn{1}{c}{\texttt{Lin}} & \multicolumn{1}{c}{\texttt{aqD}} & \multicolumn{1}{c}{\texttt{aqA\textsuperscript{*}}} \\\hline
\endfirsthead
\hline
\endhead
\hline
\multicolumn{20}{r}{{Continued on next page}}
\endfoot
\hline\hline
\caption{Results for quadratic cost variation on Erdős-Rényi random graphs}\label{tbl:lambda-erdos}
\endlastfoot
1 & 26.316 & 0.303 & 0.116 && 25.662 & 0.296 & 0.222 && 25.670 & 0.396 & 0.110 && 25.646 & 0.428 & 0.077 && 25.597 & 0.458 & 0.079 \\
2 & 25.624 & 0.315 & 0.108 && 25.622 & 0.304 & 0.097 && 25.648 & 0.141 & 0.103 && 25.587 & 0.089 & 0.120 && 25.494 & 0.079 & 0.126 \\
3 & 25.931 & 0.399 & 0.011 && 25.550 & 0.226 & 0.011 && 25.641 & 0.236 & 0.022 && 25.575 & 0.146 & 0.103 && 25.511 & 0.120 & 0.132 \\
4 & 25.577 & 0.007 & 0.023 && 25.509 & 0.011 & 0.072 && 25.603 & 0.115 & 0.051 && 25.536 & 0.173 & 0.070 && 25.575 & 0.159 & 0.069 \\
5 & 25.614 & 0.069 & 0.041 && 25.513 & 0.125 & 0.049 && 25.705 & 0.147 & 0.043 && 25.462 & 0.074 & 0.059 && 25.418 & 0.067 & 0.062 \\
6 & 25.633 & 0.130 & 0.024 && 25.552 & 0.201 & 0.086 && 25.684 & 0.218 & 0.172 && 25.588 & 0.256 & 0.085 && 25.459 & 0.253 & 0.089 \\
7 & 25.974 & 0.530 & 0.055 && 25.530 & 0.238 & 0.112 && 25.583 & 0.154 & 0.081 && 25.533 & 0.225 & 0.069 && 25.481 & 0.168 & 0.067 \\
8 & 25.647 & 0.474 & 0.019 && 25.583 & 0.249 & 0.050 && 25.628 & 0.187 & 0.049 && 25.600 & 0.222 & 0.051 && 25.467 & 0.244 & 0.056 \\
\end{longtable}

\setlength{\tabcolsep}{4.5pt}
\begin{longtable}{crrrrrrrrrrrrrrrrrrr}\hline
\multirow{2}{*}{Inst} & \multicolumn{3}{c}{$\lambda = 1$} && \multicolumn{3}{c}{$\lambda = 5$} && \multicolumn{3}{c}{$\lambda = 50$} && \multicolumn{3}{c}{$\lambda = 500$} && \multicolumn{3}{c}{$\lambda = 1000$} \\\cline{2-4}\cline{6-8}\cline{10-12}\cline{14-16}\cline{18-20}
& \multicolumn{1}{c}{\texttt{Lin}} & \multicolumn{1}{c}{\texttt{aqD}} & aqA* && \multicolumn{1}{c}{\texttt{Lin}} & \multicolumn{1}{c}{\texttt{aqD}} & \multicolumn{1}{c}{\texttt{aqA\textsuperscript{*}}} && \multicolumn{1}{c}{\texttt{Lin}} & \multicolumn{1}{c}{\texttt{aqD}} & \multicolumn{1}{c}{\texttt{aqA\textsuperscript{*}}} && \multicolumn{1}{c}{\texttt{Lin}} & \multicolumn{1}{c}{\texttt{aqD}} & \multicolumn{1}{c}{\texttt{aqA\textsuperscript{*}}} && \multicolumn{1}{c}{\texttt{Lin}} & \multicolumn{1}{c}{\texttt{aqD}} & \multicolumn{1}{c}{\texttt{aqA\textsuperscript{*}}} \\\hline
\endfirsthead
\hline
\endhead
\hline
\multicolumn{20}{r}{{Continued on next page}}
\endfoot
\hline\hline
\caption{Results for quadratic cost variation on Configuration Model random graphs}\label{tbl:lambda-conf-model}
\endlastfoot
1 & 19.540 & 0.407 & 0.685 && 19.769 & 2.717 & 0.854 && 22.052 & 8.482 & 2.943 && 20.673 & 8.694 & 3.519 && 20.501 & 8.790 & 3.553 \\
2 & 18.323 & 5.253 & 0.613 && 18.197 & 4.354 & 0.631 && 18.669 & 1.650 & 0.739 && 19.390 & 1.602 & 0.795 && 18.482 & 1.609 & 0.801 \\
3 & 19.198 & 0.304 & 0.635 && 19.913 & 0.967 & 0.949 && 19.079 & 2.749 & 2.362 && 20.097 & 2.546 & 2.975 && 20.010 & 2.541 & 2.996 \\
4 & 18.404 & 5.164 & 0.707 && 18.379 & 2.041 & 1.369 && 19.622 & 3.002 & 2.120 && 20.231 & 3.581 & 2.480 && 20.255 & 3.629 & 2.500 \\
5 & 19.177 & 0.316 & 0.654 && 18.369 & 0.434 & 0.801 && 18.330 & 1.555 & 3.248 && 18.593 & 2.256 & 3.464 && 19.336 & 2.331 & 3.550 \\
6 & 18.388 & 0.397 & 0.708 && 18.352 & 1.339 & 1.263 && 18.237 & 3.124 & 5.442 && 18.398 & 3.242 & 5.610 && 18.188 & 3.282 & 5.677 \\
7 & 19.911 & 1.781 & 0.624 && 19.115 & 0.652 & 0.827 && 22.074 & 0.319 & 1.656 && 22.462 & 0.313 & 2.070 && 22.110 & 0.327 & 2.103 \\
8 & 19.349 & 0.681 & 0.648 && 19.266 & 0.423 & 1.697 && 18.711 & 0.920 & 5.101 && 18.763 & 1.224 & 5.044 && 19.521 & 1.234 & 5.047 \\
\end{longtable}
\end{landscape}

\begin{table}[htbp]
\centering
\begin{tabular}{rrrrrrrr}\hline
\multicolumn{1}{c}{Nodes} & \multicolumn{1}{c}{Arcs} & \multicolumn{1}{c}{q-Arcs} &\multicolumn{1}{c}{\texttt{Lin}} & \multicolumn{1}{c}{\texttt{aqD}} & \multicolumn{1}{c}{\texttt{aqA\textsuperscript{*}}} & \multicolumn{1}{c}{\texttt{aqD+}} & \multicolumn{1}{c}{\texttt{aqA\textsuperscript{*}+}} \\\hline
10,000 & 39,402 & 624,492 & 0.640 & 0.233 & 0.026 & 0.278 & 0.144 \\
40,000 & 158,802 & 2,528,892 & 2.247 & 1.423 & 0.078 & 1.436 & 0.074 \\
90,000 & 358,202 & 5,713,292 & 5.112 & 3.933 & 0.173 & 4.226 & 0.290 \\
160,000 & 637,602 & 10,177,692 & 10.488 & 6.772 & 0.391 & 7.502 & 0.442 \\
250,000 & 997,002 & 15,922,092 & 16.258 & 14.060 & 0.760 & 14.620 & 0.733 \\
360,000 & 1,436,402 & 22,946,492 & 24.952 & 25.548 & 0.905 & 26.659 & 0.894 \\
490,000 & 1,955,802 & 31,250,892 & 56.910 & 34.480 & 1.178 & 39.034 & 1.203 \\
640,000 & 2,555,202 & 40,835,292 & 46.363 & 41.549 & 1.449 & 43.087 & 1.485 \\
810,000 & 3,234,602 & 51,699,692 & 82.123 & 72.122 & 1.898 & 74.534 & 1.884 \\
1,000,000 & 3,994,002 & 63,844,092 & 98.199 & 87.963 & 2.251 & 91.505 & 2.321 \\\hline
\multicolumn{3}{r}{Average} & 34.329 & 28.808 & 0.911 & 30.288 & 0.947 \\\hline
\end{tabular}
\caption{Results for Grid Graphs}
\label{tab:grid-graphs}
\end{table}

\begin{table}[htbp]
\centering
\begin{tabular}{rrrrrrrr}\hline
\multicolumn{1}{c}{Nodes} & \multicolumn{1}{c}{Arcs} & \multicolumn{1}{c}{q-Arcs} &\multicolumn{1}{c}{\texttt{Lin}} & \multicolumn{1}{c}{\texttt{aqD}} & \multicolumn{1}{c}{\texttt{aqA\textsuperscript{*}}} & \multicolumn{1}{c}{\texttt{aqD+}} & \multicolumn{1}{c}{\texttt{aqA\textsuperscript{*}+}} \\\hline
2,250,000 & 9,000,000 & 144,000,000 & 269.135 & 248.230 & 8.709 & 259.803 & 9.005 \\
3,062,500 & 12,250,000 & 196,000,000 & 415.198 & 469.030 & 11.842 & 483.949 & 11.731 \\
4,000,000 & 16,000,000 & 256,000,000 & 629.290 & 598.876 & 14.552 & 618.168 & 14.319 \\
5,062,500 & 20,250,000 & 324,000,000 & -- & 696.306 & 17.290 & 720.809 & 16.890 \\
6,250,000 & 25,000,000 & 400,000,000 & -- & 1,280.275 & 21.897 & 1,297.115 & 20.518 \\
7,562,500 & 30,250,000 & 484,000,000 & -- & 1,521.298 & 24.465 & 1,547.314 & 23.035 \\
9,000,000 & 36,000,000 & 576,000,000 & -- & 1,725.129 & 36.795 & 1,756.913 & 34.938 \\
10,562,500 & 42,250,000 & 676,000,000 & -- & 1,867.907 & 42.160 & 1,902.560 & 39.946 \\
12,250,000 & 49,000,000 & 784,000,000 & -- & -- & -- & 3,420.110 & 47.544 \\
14,062,500 & 56,250,000 & 900,000,000 & -- & -- & -- & 3,912.064 & 52.866 \\
16,000,000 & 64,000,000 & 1,024,000,000 & -- & -- & -- & 4,356.278 & 58.186 \\\hline
\multicolumn{3}{r}{Average} & -- & -- & -- & 1,843.189 & 29.907 \\\hline
\end{tabular}
\caption{Results for the stress test using Grid Graphs}
\label{tab:stress}
\end{table}